\documentclass[12pt, reqno]{amsart}
\usepackage{amsmath, amsthm, amscd, amsfonts, amssymb, graphicx, color}
\usepackage[bookmarksnumbered, colorlinks, plainpages]{hyperref}
\hypersetup{colorlinks=true,linkcolor=red, anchorcolor=green, citecolor=cyan, urlcolor=red, filecolor=magenta, pdftoolbar=true}

\newtheorem{theorem}{Theorem}[section]
\newtheorem{lemma}[theorem]{Lemma}

\theoremstyle{definition}
\newtheorem{definition}[theorem]{Definition}

\newtheorem{assumption}[theorem]{Assumption}

\theoremstyle{remark}
\newtheorem{remark}[theorem]{Remark}
\numberwithin{equation}{section}

\def\rr{{\mathbb R}}
\def\rn{{{\rr}^n}}
\def\zz{{\mathbb Z}}
\def\cc{{\mathbb C}}
\def\nn{{\mathbb N}}

\def\ca{{\mathcal A}}

\def\cd{{\mathcal D}}

\def\cf{{\mathcal F}}

\def\cm{{\mathcal M}}

\def\cp{{\mathcal P}}

\def\fz{\infty}
\def\az{\alpha}
\def\bz{\beta}

\def\bgz{{\Gamma}}

\def\lz{\lambda}

\def\tz{\theta}

\def\vz{\varphi}

\def\lf{\left}
\def\r{\right}

\def\hs{\hspace{0.25cm}}
\def\ls{\lesssim}

\def\noz{\nonumber}

\def\st{\subset}
\def\com{\complement}

\def\supp{\mathop\mathrm{\,supp\,}}

\def\div{\mathop\mathrm{div}}
\def\essinf{\mathop\mathrm{\,ess\,inf\,}}
\def\esup{\mathop\mathrm{\,ess\,sup\,}}

\def\vp{{L^{p(\cdot)}(\rn)}}

\def\vhp{H_L^{p(\cdot)}(\rn)}

\def\uz{\underline}

\def\dydt{\,\frac{dy\,dt}{t^{n+1}}}

\begin{document}

\setcounter{page}{1}

\title[Schr\"{o}dinger Type Operators]{Some Estimates of Schr\"{o}dinger Type Operators
on Variable Lebesgue and Hardy Spaces}

\author[J. Zhang \MakeLowercase{and} Z. Liu]
{Junqiang Zhang$^{1*}$ \MakeLowercase{and} Zongguang Liu$^{2}$}

\address{$^{1}$School of Science, China University of Mining and Technology-Beijing,
Beijing 100083, People's Republic of China}
\email{\textcolor[rgb]{0.00,0.00,0.84}{jqzhang@cumtb.edu.cn;
zhangjunqiang@mail.bnu.edu.cn}}

\address{$^{2}$School of Science, China University of Mining and Technology-Beijing,
Beijing 100083, People's Republic of China}
\email{\textcolor[rgb]{0.00,0.00,0.84}{liuzg@cumtb.edu.cn}}


\let\thefootnote\relax\footnote{Copyright 2018 by the Tusi Mathematical Research Group.}



\date{Received: xxxxxx; Revised: yyyyyy; Accepted: zzzzzz.
\newline \indent $^{*}$Corresponding author}

\begin{abstract}
In this article, the authors consider the Schr\"{o}dinger type
operator $L:=-\div(A\nabla)+V$ on $\mathbb{R}^n$ with $n\geq 3$,
where the matrix $A$ satisfies uniformly elliptic condition and the nonnegative potential
$V$ belongs to the reverse H\"{o}lder class $RH_q(\mathbb{R}^n)$
with $q\in(n/2,\,\infty)$.
Let $p(\cdot):\ \mathbb{R}^n\to(0,\,\infty)$ be a variable exponent function
satisfying the globally $\log$-H\"{o}lder continuous condition.
When $p(\cdot):\ \mathbb{R}^n\to(1,\,\infty)$, the authors prove that the operators
$VL^{-1}$, $V^{1/2}\nabla L^{-1}$ and $\nabla^2L^{-1}$ are
bounded on variable Lebesgue space $L^{p(\cdot)}(\mathbb{R}^n)$.
When $p(\cdot):\ \mathbb{R}^n\to(0,\,1]$, the authors introduce the
variable Hardy space $H_L^{p(\cdot)}(\mathbb{R}^n)$, associated to $L$,
and show that $VL^{-1}$, $V^{1/2}\nabla L^{-1}$ and $\nabla^2L^{-1}$ are
bounded from $H_L^{p(\cdot)}(\mathbb{R}^n)$ to $L^{p(\cdot)}(\mathbb{R}^n)$.
\end{abstract} \maketitle

\section{Introduction and main results}

Let $n\geq 3$ and consider the Schr\"{o}dinger operator $-\Delta+V$
on the Euclidean space $\rn$, where $\Delta:=\sum_{j=1}^n\frac{\partial^2}{\partial x_j^2}$
denotes the Laplacian operator on $\rn$ and $V$ is a nonnegative potential.
If $V$ is a nonnegative polynomial on $\rn$, Zhong \cite{z93} showed that
the operators $\nabla(-\Delta+V)^{-1/2}$, $\nabla^2(-\Delta+V)^{-1}$
and $\nabla(-\Delta+V)^{-1}\nabla$ are the classical Calder\'{o}n-Zygmund
operators which are bounded on $L^p(\rn)$ for any $p\in(1,\,\fz)$.

Then, in 1995, Shen \cite{sh} proved that if $V$ belongs to the
\emph{reverse H\"{o}lder class} $RH_q(\rn)$ with $q\in[n,\,\fz)$,
denoted by $V\in RH_q(\rn)$,
then $\nabla(-\Delta+V)^{-1/2}$, $(-\Delta+V)^{-1/2}\nabla$ and
$\nabla(-\Delta+V)^{-1}\nabla$ are Calder\'{o}n-Zygmund operators.
Recall that a nonnegative measurable function $V$ on $\rn$ is said to
belong to the reverse H\"older class $RH_q(\rn)$, $q\in[1,\fz]$,
if $V\in L^q_{\rm loc}(\rn)$ and there exists a positive constant $C$ such that,
for any ball $B\st\rn$,
\begin{align*}
\lf\{\frac{1}{|B|}\int_B [V(x)]^q\,dx\r\}^{1/q}\le C\frac{1}{|B|}\int_B V(x)\,dx,
\end{align*}
where we replace $\{\frac{1}{|B|}\int_B [V(x)]^q\,dx\}^{1/q}$ by
$\|V\|_{L^\infty(B)}$ when $q=\infty$.
If $V\in RH_q(\rn)$ with $q\in[n/2,\,\fz)$, Shen \cite{sh} also obtained
the $L^p(\rn)$-boundedness of $V(-\Delta+V)^{-1}$, $V^{1/2}\nabla(-\Delta+V)^{-1}$
and $\nabla^2(-\Delta+V)^{-1}$ for any $p\in(1,\,p_0)$, where $p_0\in(1,\,\fz)$ is
a constant which may depend on $n$ and $q$.
Noticing that if $V$ is a nonnegative polynomial on $\rn$,
then $V\in RH_\fz(\rn)\st RH_q(\rn)$ for any $q\in[1,\,\fz)$
(see \cite[p.\,516]{sh}), hence, Shen \cite{sh} generalized the results in \cite{z93}.
Moreover, for the weighted $L^p(\rn)$ boundedness of these operators,
we refer the reader to \cite{ly16,t15}.

For $p\in(0,\,1]$, it is well-known that many classical operators are bounded on
Hardy spaces $H^p(\rn)$, but not on $L^p(\rn)$,
for example, the Riesz transforms $\nabla(-\Delta)^{-1/2}$ and $\nabla^2(-\Delta)^{-1}$
(see \cite{s93}).
However, when working with some differential operators other than
the Laplacian operator, the classical Hardy spaces $H^p(\rn)$
are not suitable any more, since $H^p(\rn)$ is intimately connected
with the Laplacian operator.
This motivates people to develop a theory of Hardy spaces $H_L^p(\rn)$
associated with different operators $L$. This topic has attracted a lot of
attention in the last decades, which can be found in \cite{adm05,dy051,hm09,y08}.
In particular, it is showed in \cite{hlmmy11,jy11} that
the Riesz transform $\nabla(-\Delta+V)^{-1/2}$
is bounded from the Hardy space $H_{-\Delta+V}^p(\rn)$
to $L^p(\rn)$ for $p\in(0,\,1]$ and bounded on $H^p(\rn)$ for $p\in(\frac{n}{n+1},\,1]$,
where $V$ is a nonnegative potential on $\rn$.
Moreover, as a generalization of the results in \cite{sh} for $p\le 1$,
F. K. Ly \cite{ly14} proved that, for any $p\in(0,\,1]$,
the operators $\nabla^2(-\Delta+V)^{-1}$ and $V(-\Delta+V)^{-1}$
are bounded from $H_{-\Delta+V}^p(\rn)$ to $L^p(\rn)$
and, for any $p\in(\frac{n}{n+1},\,1]$, $\nabla^2(-\Delta+V)^{-1}$ is
bounded from $H_{-\Delta+V}^p(\rn)$ to $H^p(\rn)$, where $V\in RH_q(\rn)$ with
$q>\max\{n/2,\,2\}$.
Moreover, Cao et al. \cite{ccyy14} introduced Musielak-Orlicz-Hardy space
$H_{\varphi,\,-\Delta+V}(\rn)$ and, via establishing its atomic
decomposition, they obtained the boundedness of $V(-\Delta+V)^{-1}$
and $\nabla^2(-\Delta+V)^{-1}$
on $H_{\varphi,\,-\Delta+V}(\rn)$,
where $\vz:\ \rn\times[0,\fz)\to[0,\fz)$ is a Musielak-Orlicz function.
Observe that the Musielka-Orlicz-Hardy space is a more generalized
space which unifies the Hardy space, the weighted Hardy space,
the Orlicz-Hardy space and the weighted Orlicz-Hardy space.

In this paper, we consider the Schr\"{o}dinger type operator
\begin{equation}\label{eq o}
L:=-\div(A\nabla)+V\ \ \text{on}\ \ \rn,\,n\geq 3,
\end{equation}
where $V$ is a nonnegative potential and $A:=\{a_{ij}\}_{1\le i,j\le n}$
is a matrix of measurable functions satisfying the following conditions:
\begin{assumption}\label{as 1}
There exists a constant $\lz\in(0,1]$ such that,
for any $x,\,\xi\in\rn$,
$$a_{ij}(x)=a_{ji}(x)\quad \text{and}\quad
\lz|\xi|^2\le\sum_{i,j=1}^n a_{ij}(x)\xi_i\xi_j\le\lz^{-1}|\xi|^2.$$
\end{assumption}

\begin{assumption}\label{as 2}
There exist constants $\alpha\in(0,1]$ and $K\in(0,\,\fz)$ such that,
for any $i,j\in\{1,\ldots,n\}$,
$$\|a_{ij}\|_{C^\az(\rn)}\le K,$$
where $C^\az(\rn)$ denotes the set of all functions $f$
satisfying the \emph{$\alpha$-H\"{o}lder contition}
$$\|f\|_{C^\az(\rn)}:=\sup_{x,\,y\in\rn,\,x\neq y}\frac{|f(x)-f(y)|}{|x-y|^\az}<\fz.$$
\end{assumption}

\begin{assumption}\label{as 3}
There exists a constant $\alpha\in(0,1]$ such that,
for any $i,j\in\{1,\ldots,n\}$, $x\in\rn$ and $z\in\zz^n$,
$$a_{ij}\in C^{1+\az}(\rn),\quad a_{ij}(x+z)=a_{ij}(x)\quad\text{and}\quad
\sum_{k=1}^n\partial_k(a_{ij}(\cdot))(x)=0.$$
\end{assumption}

For the Schr\"{o}dinger type operator $L=-\div(A\nabla)+V$,
under the assumption that $V\in RH_\fz(\rn)$ and
$A$ satisfies some of Assumptions \ref{as 1}, \ref{as 2} and \ref{as 3} below,
Kurata and Sugano \cite{ks00} established the boundedness of
$VL^{-1}$, $V^{1/2}\nabla L^{-1}$ and $\nabla^2L^{-1}$ on the
weighted Lebesgue spaces and Morrey spaces.
Recently, motivated by \cite{ccyy14} and \cite{ks00},
Yang \cite{y14} considered the boundedness of $VL^{-1}$,
$V^{1/2}\nabla L^{-1}$ and $\nabla^2L^{-1}$ on the Musielak-Orlicz-Hardy space
$H_{\varphi,\,L}(\rn)$ associated with $L$.

On the other hand, it is well known that
the \emph{variable Lebesgue space $\vp$}
is a generalization of classical Lebesgue space,
via replacing the constant exponent $p$ by a variable exponent function
$p(\cdot):\ \rn\to(0,\,\fz)$,
which consists of all measurable functions $f$ on $\rn$ such that,
for some $\lz\in(0,\,\fz)$,
$\int_\rn [|f(x)|/\lz]^{p(x)}\,dx<\fz$,
equipped with the \emph{Luxemburg} (or known as the \emph{Luxemburg-Nakano}) {(quasi-)norm}
\begin{equation}\label{eq norm}
\|f\|_{\vp}:=\inf\lf\{\lz\in(0,\,\fz):\ \int_\rn
\lf[\frac{|f(x)|}{\lz}\r]^{p(x)}\,dx\le 1\r\}.
\end{equation}
The study of $L^p(\rn)$ originated
from Orlicz \cite{or31} in 1931.
Since then, much attention are paid to the
study of variable function spaces.
For a detailed history of this topic, we
refer the reader to the monographs \cite{cf13,dhhr11}.
As a generalized of the classical Hardy space,
the variable Hardy space $H^{p(\cdot)}(\rn)$ is
naturally considered and becomes
an active research topic in harmonic analysis, see, for example,
\cite{cw14,ns12,s13}.
Very recently, Yang et al. \cite{yz17,yzo15}
studied the variable Hardy spaces associated with different operators.

Let $L=-\div(A\nabla)+V$ be as in \eqref{eq o},
where $V$ is a nonnegative potential on $\rn$
with $n\geq 3$ and belongs to the reverse H\"{o}lder class
$RH_q(\rn)$ for some $q\in(n/2,\fz)$.
In this article, motivated by \cite{ly14,y14,yz17}, we consider the boundedness of
$VL^{-1}$, $V^{1/2}\nabla L^{-1}$ and $\nabla^2 L^{-1}$ on
variable Lebesgue spaces $L^{p(\cdot)}(\rn)$ and variable Hardy spaces $H_L^{p(\cdot)}(\rn)$
associated with $L$ (see Definition \ref{def vhp} below).
To state the main results, we first
recall some notation and definitions.

Let $\cp(\rn)$ be the set of all the measurable functions $p(\cdot):\ \rn\to (0,\,\fz)$
satisfying
\begin{equation}\label{eq var}
p_-:=\essinf_{x\in\rn}p(x)>0\ \ \text{and}\ \ p_+:=\esup_{x\in\rn}p(x)<\fz.
\end{equation}
A function $p(\cdot)\in\cp(\rn)$ is called a \emph{variable exponent function on $\rn$}.
For any $p(\cdot)\in\cp(\rn)$ with $p_-\in(1,\,\fz)$,
we define $p'(\cdot)\in\cp(\rn)$ by
\begin{align}\label{eq dual}
\frac{1}{p(x)}+\frac{1}{p'(x)}=1\ \ \ \text{for all}\ \ x\in\rn.
\end{align}
The function $p'$ is called the \emph{dual variable exponent} of $p$.

Recall that a measurable function $g\in\cp(\rn)$ is said to be
\emph{globally $\log$-H\"{o}lder continuous},
denoted by $g\in C^{\log}(\rn)$, if there exist constants $C_1,\,C_2\in(0,\,\fz)$
and $g_\fz\in\mathbb{R}$ such that, for any $x,\,y\in\rn$,
\begin{equation*}
|g(x)-g(y)|\le \frac{C_1}{\log(e+1/|x-y|)}
\end{equation*}
and
\begin{equation*}
|g(x)-g_\fz|\le \frac{C_2}{\log(e+|x|)}.
\end{equation*}

The following theorem is the first main result of
this article, which establishes the boundedness of
$VL^{-1}$, $V^{1/2}\nabla L^{-1}$ and $\nabla^2L^{-1}$ on $\vp$.
\begin{theorem}\label{thm-3}
Let $p(\cdot)\in C^{\log}(\rn)$, $L$ be as in \eqref{eq o} and $V\in RH_q(\rn)$ with
$q\in(n/2,\,\fz)$.
\begin{enumerate}
\item[(i)] Assume that $A$ satisfies Assumption \ref{as 1}. If $1<p_-\le p_+<q$,
then there exists a positive constant $C$ such that, for any $f\in\vp$,
\begin{align}\label{eq 2.14}
\lf\|VL^{-1}(f)\r\|_{\vp}\le C\|f\|_{\vp}.
\end{align}

\item[(ii)] Assume that $A$ satisfies Assumptions \ref{as 1}
and \ref{as 2}. If $1<p_-\le p_+<p_0$, where
\begin{equation}\label{eq 4.z}
p_0:=\left\{
\begin{aligned}
\frac{2qn}{3n-2q}\quad &\text{if}\ q\in(n/2,n);\\
2q\qquad\quad &\text{if}\ q\in[n,\,\fz),\\
\end{aligned}
\right.
\end{equation}
then there exists a positive constant $C$ such that, for any $f\in\vp$,
\begin{align}\label{eq 2.13}
\lf\|V^{1/2}\nabla L^{-1}(f)\r\|_{\vp}\le C\|f\|_{\vp}.
\end{align}

\item[(iii)]
Assume that $A$ satisfies Assumptions \ref{as 1}, \ref{as 2} and \ref{as 3}.
If $1<p_-\le p_+<q$, then there exists a positive constant $C$ such that, for any $f\in\vp$,
\begin{align*}
\lf\|\nabla^2L^{-1}(f)\r\|_{\vp}\le C\|f\|_{\vp}.
\end{align*}
\end{enumerate}
\end{theorem}

\begin{remark}
In particular, if $p(\cdot)\equiv p$ is a constant exponent,
then Theorem \ref{thm-3} coincides with \cite[Lemmas 3.2 and 3.3]{y14}.
\end{remark}

The proof of Theorem \ref{thm-3} is in Subsection \ref{s-4.1}.
We prove it by making use of some known results in \cite{ks00}, the fact that
the Hardy-Littlewood operator is bounded on $\vp$ (see Lemma \ref{lem 1-3} below)
and the extrapolation theorem for $\vp$ (see Lemma \ref{lem extro} below).

The \emph{quadratic operator} $S_L$, associated to $L$, is defined by
setting, for any $f\in L^2(\rn)$ and $x\in\rn$,
\begin{equation}\label{eq quadratic}
S_L(f)(x):=\lf[\iint_{\bgz(x)}\lf|t^2Le^{-t^2L}(f)(y)\r|^2\dydt\r]^{1/2},
\end{equation}
where $\Gamma(x):=\{(y,\,t)\in\rn\times(0,\,\fz):\ |y-x|<t\}$
denotes the cone with vertex $x\in\rn$.

\begin{definition}\label{def vhp}
Let $p(\cdot)\in \cp(\rn)$ satisfy $p_+\in(0,\,1]$
and $L$ be an operator as in \eqref{eq o}.
The \emph{variable Hardy space} $H_{L}^{p(\cdot)}(\rn)$ is defined as the
completion of the space
\begin{align*}
\lf\{f\in L^2(\rn):\ \|S_L(f)\|_{\vp}<\fz\r\}
\end{align*}
with respect to the \emph{quasi-norm} $\|f\|_{H_L^{p(\cdot)}(\rn)}:=\|S_L(f)\|_{\vp}$.
\end{definition}

The following theorem establishes the boundedness of $VL^{-1}$, $V^{1/2}\nabla L^{-1}$ and
$\nabla^2L^{-1}$ on $\vhp$.
\begin{theorem}\label{thm-4}
Let $L$ be as in \eqref{eq o} with $A$ satisfying Assumptions \ref{as 1}, \ref{as 2}
and \ref{as 3}. Assume that $V\in RH_q(\rn)$ with $q\in(\max\{n/2,2\},\fz)$
and $p(\cdot)\in C^{\log}(\rn)$.
If $0<p_-\le p_+\le 1$, then there exists a positive constant $C$ such that,
for any $f\in\vhp$,
\begin{equation}\label{eq-t2}
\lf\|VL^{-1}(f)\r\|_{\vp}\le C\|f\|_{\vhp},
\end{equation}
\begin{equation}\label{eq-t3}
\lf\|V^{1/2}\nabla L^{-1}(f)\r\|_{\vp}\le C\|f\|_{\vhp}
\end{equation}
and
\begin{equation}\label{eq-t1}
\lf\|\nabla^2L^{-1}(f)\r\|_{\vp}\le C\|f\|_{\vhp}.
\end{equation}
\end{theorem}

\begin{remark}\label{rem-7}
\begin{enumerate}
\item[(i)]
Noticing that $n\geq 3$,
we point out that the assumption $q\in(\max\{n/2,2\},\fz)$ guarantees that
$q>2$ and $p_0>2$, where $p_0$ is as in \eqref{eq 4.z}. By this and
Theorem \ref{thm-3} (or \cite[Lemmas 3.2 and 3.3]{y14}),
we know that $VL^{-1}$, $V^{1/2}\nabla L^{-1}$ and $\nabla^2L^{-1}$ are all bounded on $L^2(\rn)$.
Based on this, we prove Theorem \ref{thm-4} by estimating these operators acting on the
single atom appearing in the atomic decomposition of $\vhp$ (see \eqref{eq atom} below).
Observing that the atomic sums in \eqref{eq atom} converge in $L^2(\rn)$,
this is why we assume $q\in(\max\{n/2,2\},\fz)$ here.

\item[(ii)] When $p(\cdot)=p\in(0,1]$ is a constant exponent
and $L=-\Delta+V$ is the Schr\"{o}dinger operator,
\eqref{eq-t2} and \eqref{eq-t1} coincide with \cite[Theorem 1.2(a)]{ly14}.
However, \eqref{eq-t3} is new.
\end{enumerate}
\end{remark}

The proof of Theorem \ref{thm-4} is in Subsection \ref{s-4.2}.
We prove it by borrowing some ideas from
\cite{ly14,ly16} and using some skills from \cite{yz17}.
The key to the proof is to establish some weighted
estimates of the spatial derivatives of the heat kernel of $\{e^{-tL}\}_{t\geq 0}$
(see Lemma \ref{lem 4.1} below).
The proof of Lemma \ref{lem 4.1} relies on the
upper bound of the heat kernel of $\{e^{-tL}\}_{t\geq 0}$ (see Lemma \ref{lem 1-2} below)
and the inequality \eqref{eq k-1} in Lemma \ref{lem k-1},
which is, in a sense, based on the known $L^p(\rn)$-boundedness of
$\nabla^2L_0^{-1}$ with $L_0:=-\div(A\nabla)$ (see \cite[Theorem B]{al91}).
We prove \eqref{eq  k-1} by applying the method used in the proof of
\cite[Lemma 2.7]{ly16}. However, to overcome the
difficulty caused by the elliptic operator $L_0=-\div(A\nabla)$ in \eqref{eq k-1},
we need to assume that the matrix $A$ satisfies the Assumption \ref{as 3},
which plays an essentially key role in
the proof of \eqref{eq k-1} as it does in the proof of \cite[Theorem B]{al91}.

We end this section by making some conventions on notation.
In this article,
we denote by $C$ a positive constant which is independent of the main parameters,
but it may vary from line to line.
We also use $C_{(\az, \bz,\ldots)}$ to denote a positive constant depending on
the parameters $\az$, $\bz$, $\ldots$.
The \emph{symbol $f\ls g$} means that $f\le Cg$.
If $f\ls g$ and $g\ls f$,  we then write $f\sim g$.
Let $\nn:=\{1,\,2,\,\ldots\}$, $\zz_+:=\nn\cup\{0\}$
and $\vec{0}_n:=(0,\,\ldots,\,0)\in\rn$.
For any measurable subset $E$ of $\rn$, we denote by $E^\com$ the
\emph{set $\rn\setminus E$}.
For any $p\in(0,\,\fz)$ and any measurable subset $E$ of $\rn$,
let $L^p(E)$ be the set of all measurable functions $f$ on $\rn$ such that
$$\|f\|_{L^p(E)}:=\lf[\int_E |f(x)|^p\,dx\r]^{1/p}<\fz.$$
For any $k\in\nn$ and $p\in(1,\,\fz)$, denote by $W^{k,\,p}(\rn)$
the usual \emph{Sobolev space} on $\rn$ equipped with the norm
$(\sum_{|\az|\le k}\|D^\az f\|_{L^p(\rn)}^p)^{1/p}$,
where $\az\in\zz_+^n$ is a multiindex and $D^\az f$ denotes the
\emph{distributional derivatives} of $f$.
We also denote by $C_c^\fz(\rn)$ the set of all infinitely differential functions with compact supports.
For any $r\in\mathbb{R}$, the \emph{symbol} $\lfloor r\rfloor$
denotes the largest integer $m$ such that $m\le r$.
For any $\mu\in(0,\,\pi)$, let
\begin{align}\label{eq 0.1}
\Sigma_\mu^0:=\lf\{z\in\cc\setminus\{0\}:\ |\arg z|<\mu\r\}.
\end{align}
For any ball
$B:=B(x_B,r_B):=\{y\in\rn:\ |x-y|<r_B\}\st\rn$
with $x_B\in\rn$ and $r_B\in(0,\,\fz)$, $\az\in(0,\,\fz)$ and $j\in\nn$,
we let $\az B:=B(x_B,\az r_B)$,
\begin{align}\label{eq-ujb}
U_0(B):=B\ \ \ \text{and}\ \ \ U_j(B):= (2^jB)\setminus (2^{j-1}B).
\end{align}
For any $p\in[1,\,\fz]$, $p'$ denotes its conjugate number,
namely, $1/p+1/p'=1$.

\section{Preliminaries}\label{s2}

In this section, we recall some notions and
results on the Schr\"{o}dinger type operator $L=-\div(A\nabla)+V$
and variable Lebesgue space $\vp$.

We first recall the definition of the auxiliary function $m(\cdot,V)$
introduced by Shen \cite[Definition 2.1]{sh} and its properties.
Let $V\in RH_q(\rn)$, $q\in(n/2,\,\fz)$, and $V\not\equiv 0$.
For any $x\in\rn$, the auxiliary function $m(x,\,V)$ is defined by
\begin{equation}\label{eq aux}
\frac{1}{m(x,\,V)}:=\sup\lf\{r\in(0,\,\fz):\ \frac{1}{r^{n-2}}\int_{B(x,r)}V(y)\,dy\le 1\r\}.
\end{equation}

For the auxiliary function $m(\cdot,V)$, we have the following Lemma \ref{lem aux1},
which is proved in \cite[Lemmas 1.2 and 1.8]{sh}.

\begin{lemma}[\cite{sh}]\label{lem aux1}
Let $m(\cdot,V)$ be as in \eqref{eq aux} and $V\in RH_q(\rn)$ with $q\in(n/2,\,\fz)$.
\begin{enumerate}
\item[(i)] Then there exist a positive constant
$C$ such that, for any $x\in\rn$ and $0<r<R<\fz$,
\begin{equation*}
\frac{1}{r^{n-2}}\int_{B(x,r)}V(y)\,dy
\le C\lf(\frac{R}{r}\r)^{\frac{n}{q}-2}
\frac{1}{R^{n-2}}\int_{B(x,R)}V(y)\,dy.
\end{equation*}

\item[(ii)] There exist positive constants $C$ and $k_0$ such that,
for any $x\in\rn$ and $R\in(0,\,\fz)$, if $Rm(x,\,V)\geq1$, then
\begin{align*}
\frac{1}{R^{n-2}}\int_{B(x,R)}V(y)\,dy\le C\lf[Rm(x,\,V)\r]^{k_0}.
\end{align*}
\end{enumerate}
\end{lemma}

The following lemma is \cite[Theorem 4]{at98}.
\begin{lemma}[\cite{at98}]\label{lem 1-1}
Let $L_0:=-\div(A\nabla)$ with $A$ satisfying Assumption \ref{as 1}
and $\{e^{-tL_0}\}_{t\geq 0}$ the heat semigroup generated by $L_0$.
Then the kernels $h_t(x,y)$ of the heat semigroup $\{e^{-tL_0}\}_{t\geq 0}$ are
continuous and there exists a constant $\az_0\in(0,1]$ such that,
for any given $\az\in(0,\az_0)$,
\begin{equation*}
|h_t(x+h,y)-h_t(x,y)|+|h_t(x,y+h)-h_t(x,y)|\le
\frac{C}{t^{n/2}}\lf[\frac{|h|}{\sqrt{t}}\r]^\az e^{-c\frac{|x-y|^2}{t}},
\end{equation*}
where $t\in(0,\,\fz)$, $x,y,h\in\rn$ with $|h|\le\sqrt{t}$
and $C,c$ are positive constants independent of $t,x,y,h$.
\end{lemma}

The following lemma is \cite[Lemma 2.6]{y14}.
\begin{lemma}[\cite{y14}]\label{lem 1-2}
Let $L$ be as in \eqref{eq o} with $A$ satisfying Assumption \ref{as 1}
and $V\in RH_q(\rn)$, $q\in(n/2,\,\fz)$.
Assume that $K_t$ is the kernel of the heat semigroup $\{e^{-tL}\}_{t\geq 0}$ and
let
\begin{align}\label{eq mu}
\mu_0:=\min\lf\{\az_0,2-\frac{n}{q}\r\},
\end{align}
where $\az_0\in(0,1]$ is as in Lemma \ref{lem 1-1}.
\begin{enumerate}
\item[(i)] For any given $N\in\nn$,
there exist positive constants $C_{(N)}$ and $c$ such that,
for any $t\in(0,\,\fz)$ and every $(x,y)\in\rn\times\rn$,
\begin{equation}\label{eq Gauss1}
0\le K_t(x,\,y)\le\frac{C_{(N)}}{t^{n/2}}e^{-c\frac{|x-y|^2}{t}}
\lf[1+\sqrt{t}m(x,\,V)+\sqrt{t}m(y,\,V)\r]^{-N}.
\end{equation}

\item[(ii)] For any given $N\in\nn$ and $\mu\in(0,\,\mu_0)$,
there exist positive constants $C_{(N,\,\mu)}$ and $c$ such that,
for any $t\in(0,\,\fz)$ and every $x,y,h\in\rn$ with $|h|\le\sqrt{t}$,
\begin{align}\label{eq Gauss2}
&|K_t(x+h,\,y)-K_t(x,\,y)|+|K_t(x,\,y+h)-K_t(x,\,y)|\noz\\
&\hs\le \frac{C_{(N,\,\mu)}}{t^{n/2}}\lf[\frac{|h|}{\sqrt{t}}\r]^\mu
e^{-c\frac{|x-y|^2}{t}}
\lf[1+\sqrt{t}m(x,\,V)+\sqrt{t}m(y,\,V)\r]^{-N}.
\end{align}
\end{enumerate}
\end{lemma}

\begin{remark}\label{rem-3}
\begin{enumerate}
\item[(i)]
For any given $N\in\nn$,
by an argument similar to that used in the proof of \cite[Corollary 6.4]{dz03},
we find that, for any $z\in\Sigma^0_{\pi/5}$,
\begin{align}\label{eq 1.0}
|K_z(x,\,y)|\le C_{(N)}\frac{1}{|z|^{n/2}}e^{-c\frac{|x-y|^2}{|z|}}
\lf[1+\sqrt{|z|}m(x,\,V)+\sqrt{|z|}m(y,\,V)\r]^{-N},
\end{align}
where $K_z(\cdot,\,\cdot)$ denotes the integral kernel of the operator $e^{-zL}$
and $C_{(N)}$ is a positive constant depending only on $N$.
From the Cauchy formula and the holomorphy of the semigroup
$\{e^{-zL}\}_{z\in\Sigma^0_{\pi/5}}$, we deduce that,
for any $k\in\nn$ and $t\in(0,\,\fz)$,
\begin{align*}
\frac{d^k}{dt^k}e^{-tL}=(-L)^k e^{-tL}
=\frac{k!}{2\pi i}\int_{|\zeta-t|=\eta t}\frac{e^{-\zeta L}}{(\zeta-t)^{k+1}}\,d\zeta,
\end{align*}
where $\eta\in(0,\,\fz)$ is small enough such that
$\{\zeta\in\cc:\ |\zeta-t|=\eta t\}\varsubsetneqq\Sigma^0_{\pi/5}$
and $\Sigma^0_{\pi/5}$ is as in \eqref{eq 0.1}.
Hence, for every $x,\,y\in\rn$, we have
\begin{align}\label{eq 1.4}
\frac{\partial^k}{\partial t^k}K_t(x,\,y)
=(-1)^k\frac{k!}{2\pi i}\int_{|\zeta-t|=\eta t}\frac{K_\zeta(x,y)}{(\zeta-t)^{k+1}}\,d\zeta.
\end{align}
From this, \eqref{eq 1.0} and the fact that $|\zeta|\sim t$ for any
$\zeta\in\{\zeta\in\cc:\ |\zeta-t|=\eta t\}$, it follows that
\begin{align}\label{eq 1.1}
\lf|\frac{\partial^k}{\partial t^k}K_t(x,\,y)\r|
&\ls\int_{|\zeta-t|=\eta t}\frac{|K_\zeta(x,y)|}{|\zeta-t|^{k+1}}\,d|\zeta|\\
&\ls\frac{1}{t^{k+n/2}}e^{-c\frac{|x-y|^2}{t}}\lf[1+m(x,\,V)+m(y,\,V)\r]^{-N},\noz
\end{align}
where the implicit positive constant depends only on $k$, $\eta$ and $N$.

\item[(ii)]
For any given $N\in\nn$ and $\mu\in(0,\,\mu_0)$,
by \cite[Lemma 17]{at98}, we know that, for any $z\in\Sigma_{\pi/5}^0$ and $y\in\rn$,
\begin{align}\label{eq 1.3}
\|K_z(\cdot,\,y)\|_{C^\mu(\rn)}\ls|z|^{-\frac{n+\mu}{2}}
\end{align}
is equivalent to, for any $f\in L^1(\rn)$,
\begin{align}\label{eq 1.2}
\lf\|e^{-zL}(f)\r\|_{C^\mu(\rn)}\ls|z|^{-\frac{n+\mu}{2}}\|f\|_{L^1(\rn)}.
\end{align}
Similar to the proof of \cite[Lemma 19]{at98},
by \eqref{eq Gauss2} and interpolation,
we obtain \eqref{eq 1.2}, namely, \eqref{eq 1.3} holds true.
From \eqref{eq 1.3} and an argument similar to
the proof \cite[Proposition 4.11]{dz03},
we further deduce that, for any given $N\in\nn$,
there exists positive constants $C_{(N,\,\mu)}$ and $c$ such that,
for any $z\in\Sigma_{\pi/5}^0$ and every $x,y,h\in\rn$ with $|h|\le\sqrt{|z|}$,
\begin{align*}
&|K_z(x+h,\,y)-K_t(x,\,y)|+|K_z(x,\,y+h)-K_t(x,\,y)|\\
&\hs\le \frac{C_{(N,\,\mu)}}{|z|^{n/2}}\lf[\frac{|h|}{\sqrt{|z|}}\r]^\mu
e^{-c\frac{|x-y|^2}{|z|}}
\lf[1+\sqrt{|z|}m(x,\,V)+\sqrt{|z|}m(y,\,V)\r]^{-N}.
\end{align*}
This, combined with \eqref{eq 1.4}, implies that,
for any $k\in\nn$,
\begin{align}\label{eq 1.5}
&\lf|\frac{\partial^k}{\partial t^k}K_t(x+h,\,y)-
\frac{\partial^k}{\partial t^k}K_t(x,\,y)\r|
+\lf|\frac{\partial^k}{\partial t^k}K_t(x,\,y+h)-
\frac{\partial^k}{\partial t^k}K_t(x,\,y)\r|\\
&\hs\le \frac{C_{(N,\,\mu)}}{t^{k+n/2}}\lf[\frac{|h|}{\sqrt{t}}\r]^\mu
e^{-c\frac{|x-y|^2}{t}}
\lf[1+\sqrt{t}m(x,\,V)+\sqrt{t}m(y,\,V)\r]^{-N}.\noz
\end{align}

\end{enumerate}
\end{remark}

\begin{remark}\label{rem-2}
By Lemma \ref{lem 1-2}(i), we know that
the heat kernel $K_t$ satisfies the Gaussian upper bound.
From this and \cite[(3.2)]{y08}, we deduce that, for any $p\in (1,\,\fz)$,
the quadratic operator $S_L$ (see \eqref{eq quadratic})
is bounded on $L^p(\rn)$.
\end{remark}

The \emph{Hardy-Littlewood maximal operator} $\cm$ is defined by setting, for any $f\in L^1_{\rm loc}(\rn)$
and $x\in\rn$,
\begin{equation}\label{eq h-l}
\cm(f)(x):=\sup_{B\ni x}\frac{1}{|B|}\int_B |f(y)|\,dy,
\end{equation}
where the supremum is taken over all balls $B$ of $\rn$ containing $x$.

The following lemma establishes the boundedness of $\cm$ on $\vp$,
which is just \cite[Theorem 4.3.8]{dhhr11} (see also \cite[Theorem 3.16]{cf13}).
\begin{lemma}[\cite{dhhr11}]\label{lem 1-3}
Let $p(\cdot)\in C^{\log}(\rn)$ and $1<p_-\le p_+<\fz$. Then there exists a positive
constant $C$ such that, for any $f\in\vp$,
$$\|\cm(f)\|_{\vp}\le C\|f\|_{\vp}.$$
\end{lemma}

The following lemma is a particular case of \cite[Lemma 2.4]{yz17},
which is is a slight variant of \cite[Lemma 4.1]{s13}.
\begin{lemma}[\cite{yz17}]\label{lem-key}
Let $\kappa\in[1,\,\fz)$,
$p(\cdot)\in C^{\log}(\rn)$, $\underline{p}:=\min\{p_-,\,1\}$ and
$r\in[1,\,\fz]\cap(p_+,\,\fz]$, where $p_-$ and $p_+$ are as in \eqref{eq var}.
Then there exists a positive constant $C$ such that, for any sequence $\{B_j\}_{j\in\nn}$
of balls in $\rn$, $\{\lz_j\}_{j\in\nn}\st\mathbb{C}$ and functions $\{a_j\}_{j\in\nn}$ satisfying
that, for any $j\in\nn$, $\supp a_j\st \kappa B_j$ and $\|a_j\|_{L^r(\rn)}\le |B_j|^{1/r}$,
\begin{align}\label{eq-key}
\lf\|\lf(\sum_{j=1}^\fz|\lz_j a_j|^{\uz{p}}\r)^{\frac{1}{\uz{p}}}\r\|_{\vp}
\le C\kappa^{n(\frac1{\uz{p}}-\frac1r)}\lf\|\lf(\sum_{j=1}^\fz|\lz_j
\chi_{B_j}|^{\uz{p}}\r)^{\frac{1}{\uz{p}}}\r\|_{\vp}.
\end{align}
\end{lemma}

For more properties of the variable Lebesgue spaces $\vp$, we refer the
reader to \cite{cf13,dhhr11}.
\begin{remark}\label{rem-1}
Let $p(\cdot)\in\cp(\rn)$.
\begin{enumerate}
\item[(i)] For any $\lz\in\mathbb{C}$ and $f\in\vp$, $\|\lz f\|_{\vp}=|\lz|\|f\|_{\vp}$.
In particular, if $p_-\in[1,\,\fz)$, then $\|\cdot\|_{\vp}$ is a norm, namely,
for any $f,\,g\in\vp$,
$$\|f+g\|_{\vp}\le \|f\|_{\vp}+\|g\|_{\vp}.$$

\item[(ii)] By the definition of $\|\cdot\|_{\vp}$ (see \eqref{eq norm}),
it is easy to see that, for any $f\in\vp$ and $s\in(0,\,\fz)$,
\begin{equation*}
\lf\||f|^s\r\|_{\vp}=\|f\|_{L^{sp(\cdot)}(\rn)}^s.
\end{equation*}

\item[(iii)] Let $p_-\in(1,\,\fz)$. Then, by \cite[Lemma 3.2.20]{dhhr11},
we find that, for any $f\in\vp$
and $g\in L^{{p'}(\cdot)}(\rn)$,
$$\int_\rn|f(x)g(x)|\,dx\le 2\|f\|_{\vp}\|g\|_{L^{{p'}(\cdot)}(\rn)},$$
where $p'(\cdot)$ is the \emph{dual variable exponent} of $p(\cdot)$,
which is defined in \eqref{eq dual}.

\item[(iv)] Let $p_-\in(1,\,\fz)$.
Then, by \cite[Theorem 2.34]{cf13}
and \cite[Corollary 3.2.14]{dhhr11}), we know that, for any $f\in\vp$,
\begin{equation*}
\frac12\|f\|_{\vp}\le \sup_{\{g\in L^{p'(\cdot)}(\rn):\ \|g\|_{L^{p'(\cdot)}(\rn)}\le 1\}}
\int_{\rn}|f(x)g(x)|\,d\mu(x)\le 2\|f\|_{\vp}.
\end{equation*}

\end{enumerate}
\end{remark}

\section{Proofs of main results}
\subsection{Proof of Theorem \ref{thm-3}}\label{s-4.1}
In this subsection, we prove Theorem \ref{thm-3}.
We begin with introducing some auxiliary lemmas.

Let $q\in[1,\,\fz)$. Recall that a non-negative and locally integrable function $w$ on $\rn$
is said to belong to the \emph{class $A_q(\rn)$ of Muckenhoupt weights}, denoted by $w\in A_q(\rn)$,
if, when $q\in(1,\fz)$,
\begin{align*}
A_q(w):=\sup_{B\st\rn}
\frac{1}{|B|}\int_B w(x)\,dx\lf\{\frac{1}{|B|}\int_B [w(x)]^{-\frac{1}{q-1}}\,dx\r\}^{q-1}<\fz
\end{align*}
or
\begin{align*}
A_1(w):=\sup_{B\st\rn}\frac{1}{|B|}\int_B w(x)\,dx\lf\{\essinf_{x\in B}w(x)\r\}^{-1}<\fz,
\end{align*}
where the suprema are taken over all balls $B$ of $\rn$.

We need the following lemma which is
called the \emph{extrapolation theorem} for $\vp$ (see, for example,
\cite[Theorem 1.3]{cfmp06}).
\begin{lemma}[\cite{cfmp06}]\label{lem extro}
Let $\cf$ be a family of pairs of measurable functions on $\rn$.
Assume that, for some $p_0\in(0,\,\fz)$ and any $w\in A_1(\rn)$,
\begin{equation*}
\int_\rn|f(x)|^{p_0}w(x)\,dx\le C_{(w)}\int_\rn|g(x)|^{p_0}w(x)\,dx\ \ \
\text{for any}\ \ (f,\,g)\in\cf,
\end{equation*}
where the positive constant $C_{(w)}$ depends only on $A_1(w)$.
Let $p(\cdot)\in C^{\log}(\rn)$ such that $p_-\in(p_0,\,\fz)$.
Then there exists a positive constant $C$ such that, for any $(f,\,g)\in\cf$,
\begin{equation*}
\|f\|_{\vp}\le C\|g\|_{\vp}.
\end{equation*}
\end{lemma}

We also need the following lemma, which is just \cite[Theorem 1.7]{ks00}
and plays a key role in the proof of Theorem \ref{thm-3}.
\begin{lemma}[\cite{ks00}]\label{lem k-3}
Let $L$ be as in \eqref{eq o}
and $(VL^{-1})^\ast$, $(V^{1/2}\nabla L^{-1})^\ast$ the usual
dual operators of $VL^{-1}$, $V^{1/2}\nabla L^{-1}$.
\begin{enumerate}
\item[(i)]
If $A$ satisfies Assumption \ref{as 1}
and $V\in RH_q(\rn)$ with $q\in(n/2,\,\fz)$,
then there exists a positive constant $C$ such that,
for any $f\in C_c^\fz(\rn)$ and $x\in\rn$,
\begin{align*}
\lf|\lf(VL^{-1}\r)^\ast(f)(x)\r|\le C\lf[\cm\lf(|f|^{q'}\r)(x)\r]^{1/q'},
\end{align*}
where $\cm$ is the Hardy-Littlewood maximal operator as in \eqref{eq h-l}
and $q'$ the conjugate number of $q$.

\item[(ii)]
If $A$ satisfies Assumptions \ref{as 1} and \ref{as 2},
and $V\in RH_q(\rn)$ with $q\in(n/2,\,\fz)$,
then there exists a positive constant $C$ such that,
for any $f\in C_c^\fz(\rn)$ and $x\in\rn$,
\begin{align*}
\lf|\lf(V^{1/2}\nabla L^{-1}\r)^\ast(f)(x)\r|\le C\lf[\cm\lf(|f|^{p_0'}\r)(x)\r]^{1/p_0'},
\end{align*}
where $p_0$ is as in \eqref{eq 4.z} and $p_0'$ the conjugate number of $p_0$.
\end{enumerate}
\end{lemma}

We are now in a position to prove Theorem \ref{thm-3}.
\begin{proof}[Proof of Theorem \ref{thm-3}]
We first prove (i). For any $f,\,g\in C_c^\fz(\rn)$,
we have
\begin{align}\label{eq 2.13x}
\lf\langle VL^{-1}(f),\,g\r\rangle
&:=\int_\rn VL^{-1}(f)(x)g(x)\,dx\noz\\
&=\int_\rn\lf[\int_\rn V(x)\Gamma(x,y)f(y)\,dy\r]g(x)\,dx\noz\\
&=\int_\rn f(y)\lf[\int_\rn V(x)\Gamma(x,y)g(x)\,dx\r]\,dy\noz\\
&=:\lf\langle f,\, \lf(VL^{-1}\r)^\ast(g)\r\rangle,
\end{align}
where $\Gamma(\cdot,\,\cdot)$ denotes the fundamental solution of $L$.
By \eqref{eq 2.13x} and Remark \ref{rem-1}(iv), we have
\begin{align}\label{eq 2.12}
\lf\|VL^{-1}(f)\r\|_{\vp}
&\ls\sup_{\|g\|_{L^{p'(\cdot)}(\rn)}\le1}\lf|\int_\rn VL^{-1}(f)(x)g(x)\,dx\r|\noz\\
&\sim\sup_{\|g\|_{L^{p'(\cdot)}(\rn)}\le1}\lf|\int_\rn f(x)\lf(VL^{-1}\r)^\ast(g)(x)\,dx\r|\noz\\
&\ls\sup_{\|g\|_{L^{p'(\cdot)}(\rn)}\le1}\|f\|_{\vp}
\lf\|\lf(VL^{-1}\r)^\ast(g)\r\|_{L^{p'(\cdot)}(\rn)}.
\end{align}
By Lemma \ref{lem k-3}(i), we know that, for any $f\in C_c^\fz(\rn)$ and $x\in\rn$,
\begin{equation}\label{eq 2.18}
\lf|\lf(VL^{-1}\r)^\ast(f)(x)\r|\ls\lf[M\lf(|f|^{q'}\r)(x)\r]^{1/q'}.
\end{equation}
Moreover, by the fact that $1<p_-\le p_+<q$,
it is easy to see that $1<q'<p'_-\le p'_+<\fz$.
Hence, $(\frac{p'(\cdot)}{q'})_->1$.
From this, \eqref{eq 2.18}, Remark \ref{rem-1}(ii) and Lemma \ref{lem 1-3},
we deduce that, for any $g\in C_c^\fz(\rn)$,
\begin{align*}
\lf\|\lf(VL^{-1}\r)^\ast(g)\r\|_{L^{p'(\cdot)}(\rn)}
&\ls\lf\|\lf[M\lf(|g|^{q'}\r)\r]^{1/q'}\r\|_{L^{p'(\cdot)}(\rn)}
\sim\lf\|M\lf(|g|^{q'}\r)\r\|_{L^{\frac{p'(\cdot)}{q'}}(\rn)}^{1/q'}\\
&\ls\lf\||g|^{q'}\r\|_{L^{\frac{p'(\cdot)}{q'}}(\rn)}^{1/q'}
\sim\|g\|_{L^{p'(\cdot)}(\rn)}.
\end{align*}
From this and \eqref{eq 2.12}, we deduce that
\eqref{eq 2.14} holds true.

For (ii), by Lemma \ref{lem k-3}(ii) and
an argument similar to that used in the
proof of (i), we know that \eqref{eq 2.13} holds true.

Next, we prove (iii). Let $L_0:=-\div(A\nabla)=L-V$.
Then, by \cite[Theorem 2.7]{ks00}, we find that, for any given $p\in (1,\,\fz)$ and $w\in A_p(\rn)$,
$\nabla^2 L_0^{-1}$ is bounded on the weighted Lebesgue space $L^p(w)$, which
is defined to be the set of all measurable function $f$ on $\rn$ such that
$$\|f\|_{L^p(w)}:=\lf[\int_\rn |f(x)|^pw(x)\,dx\r]^{1/p}<\fz.$$
From this and Lemma \ref{lem extro},
we deduce that if $1<p_-\le p_+<\fz$, then $\nabla^2L_0^{-1}$
is bounded on $\vp$. By this, \eqref{eq 2.14} and
the fact that $L=L_0+V$, we find that if $1<p_-\le p_+<q$,
then, for any $f\in\vp$,
\begin{align*}
\lf\|\nabla^2L^{-1}(f)\r\|_{\vp}
&=\lf\|\nabla^2L_0^{-1}L_0L^{-1}\r\|_{\vp}\\
&\ls\|(L-V)L^{-1}(f)\|_{\vp}\ls\|f\|_{\vp}.
\end{align*}
This finishes the proof of Theorem \ref{thm-3}.
\end{proof}

\subsection{Proof of Theorem \ref{thm-4}}\label{s-4.2}

In this subsection, we give the proof of Theorem \ref{thm-4}.
To this end, we apply the method used in \cite{ly14,ly16}.
We begin with introducing an inequality, which is an analogue of \cite[Lemma 2.7]{ly16}
and inspired by coercive estimates on semiconvex domains established in
\cite[Theorem 4.8 and Lemma 4.17]{dhmmy13}.

\begin{lemma}\label{lem k-1}
Let $p\in(1,\,\fz)$ and $L_0:=-\div(A\nabla)$ with $A$ satisfying Assumptions
\ref{as 1}, \ref{as 2} and \ref{as 3}.
Then there exists a positive constant $C$ such that,
for any $f\in W^{2,\,p}(\rn)$ and $\phi\in C_c^\fz(\rn)$,
\begin{align}\label{eq k-1}
&\lf\|\phi\lf|\nabla^2 f\r|\r\|_{L^p(\rn)}\noz\\
&\hs\le C\lf[\lf\|f\lf|\nabla^2\phi\r|\r\|_{L^p(\rn)}
+\big\||\nabla f||\nabla\phi|\big\|_{L^p(\rn)}
+\lf\|\phi L_0(f)\r\|_{L^p(\rn)}\r].
\end{align}
\end{lemma}
\begin{proof}

Indeed, for any $f\in W^{2,\,p}(\rn)$ and $\phi\in C_c^\fz(\rn)$, we have
\begin{align*}
\phi\partial_j(\partial_k f)
&=\partial_j(\phi\partial_k f)-\partial_j\phi\partial_k f
=\partial_j(\partial_k(\phi f)-(\partial_k\phi)f)-\partial_j\phi\partial_k f\\
&=\partial_j(\partial_k(\phi f))-\partial_j f\partial_k\phi-f\partial_j(\partial_k\phi)
-\partial_j\phi\partial_k f.
\end{align*}
By this and the fact that $\nabla^2L_0^{-1}$ is bounded on $L^p(\rn)$ for any $p\in(1,\,\fz)$
(see \cite[Theorem 2.7]{ks00}), we know that
\begin{align}\label{eq 4.b}
&\lf\|\phi\partial_j(\partial_k f)\r\|_{L^p(\rn)}\noz\\
&\hs\le\lf\|\nabla^2(\phi f)\r\|_{L^p(\rn)}+2\big\||\nabla f||\nabla\phi|\big\|_{L^p(\rn)}
+\lf\|f\lf|\nabla^2\phi\r|\r\|_{L^p(\rn)}\noz\\
&\hs\ls\|L_0(\phi f)\|_{L^p(\rn)}+2\big\||\nabla f||\nabla\phi|\big\|_{L^p(\rn)}
+\lf\|f\lf|\nabla^2\phi\r|\r\|_{L^p(\rn)}.
\end{align}

For $L_0(\phi f)$, it holds true that
\begin{align*}
L_0(\phi f)
&=-\div(f(A\nabla\phi)+\phi (A\nabla f))\\
&=-(A\nabla\phi)\cdot(\nabla f)+fL_0(\phi)-(A\nabla f)\cdot(\nabla\phi)
+\phi L_0(f).
\end{align*}
Thus, we obtain
\begin{align}\label{eq 4.a}
\|L_0(\phi f)\|_{L^p(\rn)}
\ls\|f L_0(\phi)\|_{L^p(\rn)}+\big\||\nabla\phi||\nabla f|\big\|_{L^p(\rn)}
+\|\phi L_0(f)\|_{L^p(\rn)}.
\end{align}
For $L_0(\phi)$, we find that
\begin{align}\label{eq 4.d}
L_0(\phi)
&=-\div(A\nabla\phi)
=-\sum_{i=1}^n\partial_i\lf(\sum_{j=1}^na_{ij}(\partial_j\phi)\r)\noz\\
&=-\sum_{i=1}^n\sum_{j=1}^n
\lf[(\partial_i a_{ij})(\partial_j\phi)+a_{ij}(\partial_i\partial_j\phi)\r].
\end{align}
By Assumption \ref{as 3}, we know that $\sum_{i=1}^n\partial_ia_{ij}=0$ for any $j\in\{1,\,\ldots,\,n\}$.
This further implies that
$
-\sum_{i=1}^n\sum_{j=1}^n(\partial_ia_{ij})(\partial_j\phi)
=-\sum_{j=1}^n\sum_{i=1}^n(\partial_ia_{ij})(\partial_j\phi)=0.
$
From this, \eqref{eq 4.d} and Assumption \ref{as 1}, we deduce that
\begin{align*}
|L_0(\phi)|
=\lf|\sum_{i=1}^n\sum_{j=1}^na_{ij}(\partial_i\partial_j\phi)\r|
\ls\lf|\nabla^2\phi\r|.
\end{align*}
This, combined with \eqref{eq 4.b} and \eqref{eq 4.a}, implies that
\eqref{eq k-1} holds true.
Hence, we complete the proof of Lemma \ref{lem k-1}.
\end{proof}

\begin{remark}
In Lemma \ref{lem k-1}, if $A=I$ is the identity matrix, namely,
$L_0=-\Delta$ is the Laplacian operator,
then Lemma \ref{lem k-1} is just \cite[Lemma 2.7]{ly16}.
\end{remark}

Next, inspired by F. K. Ly \cite[Proposition 3.3]{ly14} and \cite[Proposition 2.4]{ly16},
we introduce some weighted estimates for the spatial derivatives of the heat kernel of $\{e^{-tL}\}_{t\geq 0}$,
which play a key role in the proof of Theorem \ref{thm-4}.

\begin{lemma}\label{lem 2.2}
Assume $V\in RH_q(\rn)$ with $q\in(n/2,\,\fz)$.
Let $K_t$ be the kernel of the heat semigroup $\{e^{-tL}\}_{t\geq 0}$ and
$q^\ast:=\frac{nq}{n-q}$ if $q\in(n/2,\,n)$, or $q^\ast:=\fz$ if $q\in [n,\,\fz]$.
For any given $p\in[1,\,q^\ast)$ and $N\in\nn$, there exist positive constants
$\az$, $C$ and $c$ such that, for any $y\in\rn$ and $t\in(0,\,\fz)$,
\begin{align}\label{eq 4.20}
\lf[\int_\rn \lf|\nabla_x
\frac{\partial^k}{\partial t^k}K_t(x,\,y)\r|^p
e^{\az\frac{|x-y|^2}{t}}\,dx\r]^{1/p}
\le\frac{C}{t^{1/2+k+n/{2p'}}}\lf[1+\sqrt{t}m(y,\,V)\r]^{-N}.
\end{align}
\end{lemma}
The proof of Lemma \ref{lem 2.2} is totally similar to
those of \cite[Proposition 2.4(b)]{ly16}
\cite[Proposition 3.3]{ly14}, the details being omitted.

\begin{lemma}\label{lem 4.1}
Let $V\in RH_q(\rn)$ with $q\in(n/2,\,\infty)$
and $K_t$ the kernel of the heat semigroup $\{e^{-tL}\}_{t\geq 0}$.

{\rm (i)}
For any given $p\in(1,\,q]$, $k\in\zz_+$ and $N\in\nn$,
there exist positive constants $\az$ and $C$ such that,
for any $t\in(0,\,\fz)$ and $y\in\rn$,
\begin{align}\label{eq 4.10}
\lf[\int_\rn\lf|\nabla_x^2\frac{\partial^k}{\partial t^k}K_t(x,\,y)\r|^p
e^{\az\frac{|x-y|^2}{t}}\,dx\r]^{1/p}
\le\frac{C}{t^{1+k+n/{2p'}}}\lf[1+\sqrt{t}m(y,\,V)\r]^{-N}
\end{align}
and
\begin{align}\label{eq 4.10y}
\lf[\int_\rn\lf|V(x)\frac{\partial^k}{\partial t^k}K_t(x,\,y)\r|^p
e^{\az\frac{|x-y|^2}{t}}\,dx\r]^{1/p}
&\le\frac{C}{t^{1+k+n/{2p'}}}\noz\\
&\hs\times\lf[1+\sqrt{t}m(y,\,V)\r]^{-N}.
\end{align}

{\rm (ii)}
For any given $p\in(1,\,p_0)$ with $p_0$ as in \eqref{eq 4.z}, $k\in\zz_+$ and $N\in\nn$,
there exist positive constants $\az$ and $C$ such that,
for any $t\in(0,\,\fz)$ and $y\in\rn$,
\begin{align*}
\lf[\int_\rn\lf|[V(x)]^{1/2}\nabla_x \frac{\partial^k}{\partial t^k}K_t(x,\,y)\r|^p
e^{\az\frac{|x-y|^2}{t}}\,dx\r]^{1/p}
\le\frac{C}{t^{1+k+n/{2p'}}}\lf[1+\sqrt{t}m(y,\,V)\r]^{-N}.
\end{align*}

\end{lemma}

\begin{remark}\label{rem-5}
In particular, when $L=-\Delta+V$ is the Schr\"{o}dinger operator, Lemma \ref{lem 4.1}(i)
is proved in \cite[Proposition 3.3]{ly14}. However,
Lemma \ref{lem 4.1}(ii) is new even in this case.
\end{remark}

\begin{proof}[Proof of Lemma \ref{lem 4.1}]
We first prove (i) by considering two cases.

Case i): $k=0$.
For any $t\in(0,\,\fz)$, choose a function $\phi^t\in C_c^\fz(\rn)$ satisfying that
\begin{enumerate}
\item[(i)] $\supp \phi^t\st B(\vec{0}_n,2\sqrt{t})$,
$\phi^t(x)\equiv 1$ for any $x\in B(\vec{0}_n,\sqrt{t})$,
$|\phi^t(x)|\le 1$ for any $x\in\rn$;
\item[(ii)] there exists a positive constant $c$ such that, for any $x\in\rn$,
$|\nabla\phi^t(x)|\le c/\sqrt{t}$, $|\nabla^2\phi^t(x)|\le c/t.$
\end{enumerate}
For any $R\in(0,\,\fz)$ and $x\in\rn$, let $\phi^t_R(x):=\phi^t(x/R)$.
For any $t,\,R\in(0,\,\fz)$ and $x,\,y\in\rn$, define
\begin{align}\label{eq 4.w}
w^{t}_R(x,y):=\phi^t_R(x-y)e^{\az\frac{|x-y|^2}{t}},
\end{align}
where $\az$ is a positive constant which is determined later.
Then, by a simple calculation,
it is easy to see that $\supp w^t_R(\cdot,\,y)\st B(y,2R\sqrt{t})$
and, for any $R\in[1,\,\fz)$,
\begin{align}\label{eq 4.11}
|w_R^t(x,y)|\le e^{\az\frac{|x-y|^2}{t}},\
\lf|\nabla_x w^t_R(x,y)\r|\le \frac{c}{\sqrt{t}}e^{\az\frac{|x-y|^2}{t}}\ \
\text{and}\ \ \lf|\nabla^2_x w^t_R(x,y)\r|\le\frac{c}{t}e^{\az\frac{|x-y|^2}{t}}.
\end{align}
For any $t\in (0,\,\fz)$, $R\in[1,\,\fz)$ and $y\in\rn$, let
\begin{align*}
{\rm I}_R^t(y):=\lf\|w_R^t(\cdot,\,y)\nabla^2 K_t(\cdot,\,y)\r\|_{L^p(\rn)}.
\end{align*}
To estimate ${\rm I}_R^t(y)$, we make use of Lemma \ref{lem k-1} with $f:=K_t(\cdot,\,y)$
and $\phi:=w_R^t(\cdot,\,y)$ therein. It is obvious that $w_R^t(\cdot,\,y)\in C_c^\fz(\rn)$.
Next we show that, for any $p\in(1,\,q]$, $K_t(\cdot,\,y)\in W^{2,\,p}(\rn)$.
Indeed, by the fact $\nabla^2L^{-1}$ is bounded on $L^p(\rn)$ for any $p\in (1,\,q]$,
$q\in(n/2,\,\fz)$,
(see Theorem \ref{thm-3}(iii))
and $L(K_t(\cdot,\,y))=\frac{\partial}{\partial t}K_t(\cdot,\,y)\in L^p(\rn)$,
we find that
\begin{align*}
\lf\|\nabla^2 K_t(\cdot,\,y)\r\|_{L^p(\rn)}=\lf\|\nabla^2L^{-1}LK_t(\cdot,\,y)\r\|_{L^p(\rn)}
\ls\lf\|\frac{\partial}{\partial t}K_t(\cdot,\,y)\r\|_{L^p(\rn)}<\fz.
\end{align*}
Thus, $K_t(\cdot,\,y)\in W^{2,\,p}(\rn)$.
From this and Lemma \ref{lem k-1}, we deduce that
\begin{align*}
{\rm I}^t_R(y)
&\ls\lf\|\nabla^2 w^t_R(\cdot,\,y)|K_t(\cdot,\,y)|\r\|_{L^p(\rn)}
+\lf\|\lf|\nabla w^t_R(\cdot,\,y)\r|\lf|\nabla K_t(\cdot,\,y)\r|\r\|_{L^p(\rn)}\\
&\hs+\lf\|w^t_R(\cdot,\,y)L_0(K_t(\cdot,\,y))\r\|_{L^p(\rn)}\\
&=:{\rm I}^t_{R,1}(y)+{\rm I}^t_{R,2}(y)+{\rm I}^t_{R,3}(y).
\end{align*}

For ${\rm I}_{R,1}^t$, by \eqref{eq 4.11} and \eqref{eq Gauss1},
we conclude that, for any given $N\in\nn$,
\begin{align}\label{eq 4.18}
{\rm I}_{R,1}^t(y)
&=\lf[\int_\rn\lf|\nabla_x^2w^t_R(x,y)\r|^p|K_t(x,\,y)|^p\,dx\r]^{1/p}\noz\\
&\ls\frac{1}{t^{1+n/2}}\lf[1+\sqrt{t}m(y,\,V)\r]^{-N}
\lf[\int_\rn e^{(\az-cp)\frac{|x-y|^2}{t}}\,dx\r]^{1/p}\noz\\
&\ls\frac{1}{t^{1+n/{2p'}}}\lf[1+\sqrt{t}m(y,\,V)\r]^{-N}.
\end{align}

To estimate ${\rm I}_{R,2}^t$, by Lemma \ref{lem 2.2}(i), we know that,
for any given $N\in\nn$,
there exist positive constants $C,\az_0,c$ such that,
for any $t\in(0,\,\fz)$ and $y\in\rn$,
\begin{align*}
\lf[\int_\rn\lf|\nabla_x K_t(x,\,y)\r|^pe^{\az_0\frac{|x-y|^2}{t}}\,dx\r]^{\frac1p}
\le\frac{C}{t^{\frac12+\frac{n}{2p'}}}\lf[1+\sqrt{t}m(y,\,V)\r]^{-N}.
\end{align*}
From this and \eqref{eq 4.11}, we deduce that
\begin{align}\label{eq 4.14}
{\rm I}_{R,2}^t(y)
&=\lf[\int_\rn\lf|\nabla_x w^t_R(x,y)\r|^p\lf|\nabla_x K_t(x,\,y)\r|^p\,dx\r]^{1/p}\noz\\
&\ls\frac{1}{t^{1/2}}\lf[\int_\rn\lf|\nabla_x K_t(x,\,y)\r|^p e^{\az\frac{|x-y|^2}{t}}\,dx\r]^{1/p}\noz\\
&\ls\frac{1}{t^{1+n/{2p'}}}\lf[1+\sqrt{t}m(y,\,V)\r]^{-N},
\end{align}
where the positive constant $\az$ as in \eqref{eq 4.w} is chosen small enough
such that $\az<\az_0$.

For ${\rm I}_{R,3}^t$, by the fact that $L=L_0+V$ and
$L(K_t(\cdot,\,y))=\frac{\partial}{\partial t}K_t(\cdot,\,y)$, we find that
\begin{align}\label{eq 4.17x}
{\rm I}_{R,3}^t(y)
&=\lf[\int_\rn \lf|w_R^t(x,y)\r|^p|(L-V)(K_t(\cdot,\,y))(x)|^2\,dx\r]^{1/p}\noz\\
&\le\lf[\int_\rn \lf|w_R^t(x,y)\r|^p\lf|\frac{\partial}{\partial t}K_t(x,\,y)\r|^p\,dx\r]^{1/p}\noz\\
&\hs+\lf[\int_\rn \lf|w_R^t(x,y)\r|^p|V(x)K_t(x,\,y)|^p\,dx\r]^{1/p}\noz\\
&={\rm I}_{R,3}^{t,1}(y)+{\rm I}_{R,3}^{t,2}(y).
\end{align}
It follows from \eqref{eq 4.11} and \eqref{eq 1.1} that
\begin{align}\label{eq 4.17}
{\rm I}_{R,3}^{t,1}(y)
&\ls\frac{1}{t^{1+n/2}}\lf[1+\sqrt{t}m(y,\,V)\r]^{-N}
\lf[\int_\rn e^{p(a-c)\frac{|x-y|^2}{t}}\,dx\r]^{1/p}\noz\\
&\ls\frac{1}{t^{1+n/{2p'}}}\lf[1+\sqrt{t}m(y,\,V)\r]^{-N},
\end{align}
where the positive constant $c$ is as in \eqref{eq 1.1} and $\az$ is chosen small
enough such that $\az<c$.

For ${\rm I}_{R,3}^{t,2}$, by \eqref{eq 4.11} and \eqref{eq Gauss1},
we know that
\begin{align}\label{eq 4.15}
{\rm I}_{R,3}^{t,2}(y)
&\ls\frac{1}{t^{n/2}}\lf[1+\sqrt{t}m(y,\,V)\r]^{-N}
\lf\{\int_\rn [V(x)]^p e^{p(\az-c)\frac{|x-y|^2}{t}}\,dx\r\}^{1/p}\noz\\
&\ls\frac{1}{t^{n/2}}\lf[1+\sqrt{t}m(y,\,V)\r]^{-N}\noz\\
&\hs\times\sum_{j=0}^\fz\lf\{\int_{U_j(B(y,\sqrt{t}))}
[V(x)]^pe^{-pc_0\frac{|x-y|^2}{t}}\,dx\r\}^{1/2},
\end{align}
where $c_0:=(c-\az)\in(0,\,\fz)$ and $U_j(B(y,\,\sqrt{t}))$ is as in \eqref{eq-ujb}.
Since $V\in RH_q(\rn)$ and $p\in(1,\,q]$,
we know that $V\in RH_p(\rn)$ and there exists some $p_0\in[1,\,\fz)$ such
that $V\in A_{p_0}(\rn)$.
By this, we find that, for any $j\in\zz_+$,
\begin{align}\label{eq 4.16}
&\lf\{\int_{U_j(B(y,\sqrt{t}))}[V(x)]^pe^{-pc_0\frac{|x-y|^2}{t}}\,dx\r\}^{1/p}\noz\\
&\hs\le e^{-c_02^{2j}}|B(y,2^j\sqrt{t})|^{\frac1p}
\lf\{\frac{1}{|B(y,2^j\sqrt{t})|}\int_{B(y,2^j\sqrt{t})}[V(x)]^p\,dx\r\}^{1/p}\noz\\
&\hs\ls e^{-c_02^{2j}}|B(y,2^j\sqrt{t})|^{-\frac1{p'}}\int_{B(y,2^j\sqrt{t})}[V(x)]\,dx\noz\\
&\hs\ls e^{-c_02^{2j}}2^{-jn/p'}2^{jp_0n}\frac{1}{t^{1-\frac{n}{2p}}}
\frac{1}{t^{\frac{n}{2}-1}}\int_{B(y,\sqrt{t})} V(x)\,dx,
\end{align}
where, in the last inequality, we use the fact that $V\in A_{p_0}(\rn)$
is a doubling measure (see, for example, \cite[p.\,196]{s93}), namely, there exists a positive
constant $C$ such that, for any ball $B$ of $\rn$, $V(2B)\le C2^{p_0n}V(B)$.
If $\sqrt{t}m(y,\,V)<1$, then, by Lemma \ref{lem aux1}(i),
we know that
\begin{align*}
\frac{1}{t^{\frac{n}{2}-1}}\int_{B(y,\sqrt{t})} V(x)\,dx
\ls \lf[\sqrt{t}m(y,\,V)\r]^{2-\frac{n}{q}}\ls1.
\end{align*}
If $\sqrt{t}m(y,\,V)\geq1$, then, by Lemma \ref{lem aux1}(ii),
we have
\begin{align*}
\frac{1}{t^{\frac{n}{2}-1}}\int_{B(y,\sqrt{t})} V(x)\,dx
\ls \lf[\sqrt{t}m(y,\,V)\r]^{k_0},
\end{align*}
where $k_0\in(0,\,\fz)$ is as in Lemma \ref{lem aux1}(ii).
From this, \eqref{eq 4.15} and \eqref{eq 4.16}, it follows that
\begin{align}\label{eq 4.17y}
{\rm I}_{R,3}^{t,2}(y)
&\ls\frac{1}{t^{n/2}}\lf[1+\sqrt{t}m(y,\,V)\r]^{-N}\frac{1}{t^{1-\frac{n}{2p}}}
\lf\{1+\lf[\sqrt{t}m(y,\,V)\r]^{k_0}\r\}\noz\\
&\hs\times\sum_{j=0}^\fz e^{-c_02^{2j}}2^{-jn/p}2^{jp_0n}\noz\\
&\ls\frac{1}{t^{1+n/{2p'}}}\lf[1+\sqrt{t}m(y,\,V)\r]^{-(N-k_0)}.
\end{align}
This, combined with \eqref{eq 4.17}, implies that
\begin{equation*}
{\rm I}_{R,3}^t(y)\ls\frac{1}{t^{1+n/{2p'}}}\lf[1+\sqrt{t}m(y,\,V)\r]^{-(N-k_0)}.
\end{equation*}
By this, \eqref{eq 4.14} and \eqref{eq 4.18}, we further conclude that
\begin{align*}
\lf\|w_R^t(\cdot,\,y)\nabla^2 K_t(\cdot,\,y)\r\|_{L^p(\rn)}
={\rm I}_R^t(y)
\ls\frac{1}{t^{1+n/{2p'}}}\lf[1+\sqrt{t}m(y,\,V)\r]^{-(N-k_0)},
\end{align*}
where the implicit positive constant is independent of $R$, $t$ and $y$.
Noticing that $\supp\phi_R^t(\cdot/R)\st B(\vec{0}_n,\,2R\sqrt{t})$,
via letting $R\to\fz$, we obtain
\begin{align*}
\lf[\int_\rn \lf|\nabla_x^2K_t(x,\,y)\r|^p e^{\az\frac{|x-y|^2}{t}}\,dx\r]^{1/p}
&=\lim_{R\to\fz}{\rm I}_{R}^t(y)\\
&\ls\frac{1}{t^{1+n/{2p'}}}\lf[1+\sqrt{t}m(y,\,V)\r]^{-(N-k_0)}.
\end{align*}

For \eqref{eq 4.10y}, we have
\begin{align*}
\lf[\int_\rn\lf|V(x)\frac{\partial^k}{\partial t^k}K_t(x,\,y)\r|^p
e^{\az\frac{|x-y|^2}{t}}\,dx\r]^{1/p}
&=\lim_{R\to\fz}{\rm I}_{R,\,3}^{t,\,2}(y)\\
&\ls\frac{1}{t^{1+n/{2p'}}}[1+\sqrt{t}m(y,\,V)]^{-(N-k_0)},
\end{align*}
where ${\rm I}_{R,\,3}^{t,\,2}(y)$ is as in \eqref{eq 4.17x}
and the last inequality follows from \eqref{eq 4.17y}.
Observing $N\in\nn$ is arbitrary, we know that \eqref{eq 4.10} and
\eqref{eq 4.10y} hold true for any given $N\in\nn$.
This finishes the proof of (i) for $k=0$.

Case ii): $k\in\nn$. In this case, by \eqref{eq 1.1} and an argument similar to
that used in the proof of \cite[Proposition 3.3]{ly14},
we conclude that \eqref{eq 4.10} and \eqref{eq 4.10y} hold true,
which completes the proof of (i).

Next, we prove (ii). If $q\in(n/2,n)$, then $p\in(1,\frac{2qn}{3n-2q})$
and we could choose some positive constant $p_1\in(p,\frac{nq}{n-q})$
such that $\frac{pp_1}{2(p_1-p)}<q$.
If $q\in[n,\,\fz)$, then $p\in(1,\,2q)$
and  we could choose some positive constant $p_1\in(p,\fz)$ such that $\frac{pp_1}{2(p_1-p)}<q$.
By this, the H\"{o}lder inequality and \eqref{eq 4.20}, we have
\begin{align}\label{eq 4.16x}
&\lf\{\int_\rn\lf|[V(x)]^{1/2}\nabla_x \frac{\partial^k}{\partial t^k}K_t(x,\,y)\r|^p
e^{\az\frac{|x-y|^2}{t}}\,dx\r\}^{1/p}\noz\\
&\hs\le\lf\{\int_\rn \lf|\nabla_x \frac{\partial^k}{\partial t^k}K_t(x,\,y)\r|^{p_1}
e^{\frac{2\az p_1}{p}\frac{|x-y|^2}{t}}\,dx\r\}^{\frac1{p_1}}\noz\\
&\hs\hs\times\lf\{\int_\rn\lf([V(x)]^{p/2}e^{-\az\frac{|x-y|^2}{t}}\r)^{(p_1/p)'}\,dx
\r\}^{\frac1{p}-\frac{1}{p_1}}\noz\\
&\hs\ls\frac{1}{t^{1/2+k+n/{2p_1'}}}\lf[1+\sqrt{t}m(y,\,V)\r]^{-N}\noz\\
&\hs\hs\times\sum_{j=0}^\fz e^{-c_02^{2j}}
\lf\{\int_{U_j(B(y,\,\sqrt{t}))}[V(x)]^{\frac{pp_1}{2(p_1-p)}}
\,dx\r\}^{\frac{1}{p}-\frac1{p_1}}\noz\\
&\hs\ls\frac{1}{t^{1/2+k+n/{2p_1'}}}\lf[1+\sqrt{t}m(y,\,V)\r]^{-N}
\sum_{j=0}^\fz e^{-c_02^{2j}}\lf|B(y,\,2^j\sqrt{t})\r|^{\frac1{p}-\frac{1}{p_1}}\noz\\
&\hs\hs\times\lf\{\frac{1}{|B(y,2^j\sqrt{t})|}
\int_{B(y,\,2^j\sqrt{t})}[V(x)]^{\frac{pp_1}{2(p_1-p)}}\,dx\r\}^{\frac{1}{p}-\frac1{p_1}}.
\end{align}
By the fact that $V\in RH_q(\rn)$, we know that there exists some
$p_0\in[1,\,\fz)$ such that $V\in A_{p_0}(\rn)$.
From this, the fact that $RH_q(\rn)\st RH_{\frac{pp_1}{2(p_1-p)}}(\rn)$
and an argument similar to that used in \eqref{eq 4.16}, we deduce that
\begin{align*}
&\sum_{j=0}^\fz e^{-c_02^{2j}}\lf|B(y,\,2^j\sqrt{t})\r|^{\frac1{p}-\frac{1}{p_1}}
\lf\{\frac{1}{|B(y,\,2^j\sqrt{t})|}
\int_{B(y,\,2^j\sqrt{t})}[V(x)]^{\frac{pp_1}{2(p_1-p)}}\,dx\r\}^{\frac{1}{p}-\frac1{p_1}}\\
&\hs\ls\sum_{j=0}^\fz e^{-c_02^{2j}}\lf|B(y,\,2^j\sqrt{t})\r|^{\frac1{p}-\frac{1}{p_1}}
\lf\{\frac{1}{|B(y,\,2^j\sqrt{t})|}
\int_{B(y,\,2^j\sqrt{t})}V(x)\,dx\r\}^{1/2}\\
&\hs\ls\sum_{j=0}^\fz e^{-c_02^{2j}}2^{-jn(\frac12-\frac{p_1-p}{p_1p})}2^{jnp_0/2}
t^{-[\frac12-\frac{n(p_1-p)}{2p_1p}]}
\lf[\frac{1}{t^{\frac{n}{2}-1}}\int_{B(y,\,\sqrt{t})}V(x)\,dx\r]^{1/2}\\
&\hs\ls t^{-[\frac12-\frac{n(p_1-p)}{2p_1p}]}.
\end{align*}
This, combined with \eqref{eq 4.16x},
finishes the proof of Lemma \ref{lem 4.1}.
\end{proof}

To prove Theorem \ref{thm-4},
we also need the following atomic decomposition of $\vhp$ (see \eqref{eq atom} below),
which is established in \cite[Proposition 5.12]{yz17}.
\begin{definition}[\cite{yz17}]\label{def atom-2}
Let $p(\cdot)\in\cp(\rn)$ with $p_+\in(0,1]$ and $M\in\nn$.
A funtion $a\in L^2(\rn)$ is called an $(p(\cdot),\,M)_L$-atom,
associated with $L$, if there exists a function $b\in \cd(L^M)$ and
a ball $B:=B(x_B,\,r_B)$ of $\rn$ such that $a=L^M(b)$ and,
for any $k\in\{0,\,\cdots,\,M\}$,
\begin{enumerate}
\item[(i)]$\supp L^k(b)\st B$;
\item[(ii)]$\|(r_B^2L)^k(b)\|_{L^2(\rn)}\le r_B^{2M}|B|^{1/2}\|\chi_B\|_{L^{p(\cdot)}(\rn)}^{-1}$.
\end{enumerate}
\end{definition}

In what follows, for any $p(\cdot)\in\cp(\rn)$ with $0<p_-\le p_+\le 1$,
any sequences $\{\lz_j\}_{j\in\nn}\st\cc$ and $\{B_j\}_{j\in\nn}$ of balls in $\rn$, define
\begin{equation}\label{eq norm-2}
\ca(\{\lz_j\}_{j\in\nn},\,\{B_j\}_{j\in\nn}):=\lf\|\lf\{\sum_{j\in\nn}
\lf[\frac{|\lz_j|\chi_{B_j}}{\|\chi_{B_j}\|_{\vp}}\r]^{p_-}\r\}^{\frac1{p_-}}\r\|_{\vp}.
\end{equation}

Let $p(\cdot)\in C^{\log}(\rn)$ with $p_+\in(0,\,1]$ and
$M\in \nn\cap(\frac{n}{2}[\frac1{p_-}-\frac12],\,\fz)$.
Then, from the fact that $L=-\div(A\nabla)+V$ is a nonnegative self-adjoint operator
and \cite[Proposition 5.12]{yz17},
we deduce that, for any $f\in\vhp\cap L^2(\rn)$, there exists a sequence
$\{\lz_j\}_{j\in\nn}\st\cc$ and a family $\{a_j\}_{j\in\nn}$ of
$(p(\cdot),\,M)_L$-atoms, associated with balls $\{B_j\}_{j\in\nn}$ of $\rn$, such that
\begin{align}\label{eq atom}
f=\sum_{j=1}^\fz \lz_ja_j\ \ \text{in}\ \ L^2(\rn)\quad\text{and}\quad
\ca(\{\lz_j\}_{j\in\nn},\,\{B_j\}_{j\in\nn})\sim\|f\|_{\vhp},
\end{align}
where the implicit positive constants are independent of $f$.

The following lemma shows that the above atomic decomposition of $\vhp$ allows one to reduce
the study of the boundedness of operators on $\vhp$ to studying
their behaviours on single atoms.
\begin{lemma}\label{lem 4.2}
Let $L$ be as in \eqref{eq o}
and $p(\cdot)\in C^{\log}(\rn)$ with $p_+\in(0,\,1]$.
Suppose $T$ is a linear operator, or a positive sublinear operator, which
is bounded on $L^2(\rn)$.
Let $M\in\nn\cap(\frac n2[\frac{1}{p_-}-\frac12],\,\fz)$.
Assume that there exist positive constants $C$ and $\tz\in(n[\frac{1}{p_-}-\frac12],\,\fz)$
such that, for any $(p(\cdot),\,M)_L$-atom $a$, associated with ball $B$ of $\rn$,
and $i\in\zz_+$,
\begin{align}\label{eq 4.26}
\|T(a)\|_{L^2(U_i(B))}\le C2^{-i\tz}|B|^{\frac12}\|\chi_B\|_{\vp}^{-1}.
\end{align}
Then there exists a positive constant $C$ such that, for any $f\in\vhp$,
\begin{align}\label{eq 4.25}
\|T(f)\|_{\vp}\le C\|f\|_{\vhp}.
\end{align}
\end{lemma}
The proof of Lemma \ref{lem 4.2} is similar to that of \cite[Corollary 3.16]{yzz15},
the details being omitted.

We now prove Theorem \ref{thm-4}.
\begin{proof}[Proof of Theorem \ref{thm-4}]
We first prove \eqref{eq-t2}.
By the fact that $q>\max\{n/2,2\}$ and Theorem \ref{thm-3}(i),
we find that $VL^{-1}$ is bounded on $L^2(\rn)$.
By this and Lemma \ref{lem 4.2}, to prove \eqref{eq-t2},
it suffices to show that
there exist positive constants $C$ and $\tz\in(n[\frac{1}{p_-}-\frac12],\,\fz)$
such that, for any $(p(\cdot),\,M)_L$-atom $a$ with
$M\in\nn\cap(\frac n2[\frac{1}{p_-}-\frac12],\,\fz)$,
associated with ball $B:=B(x_B,r_B)$ of $\rn$,
and any $i\in\zz_+$,
\begin{align}\label{eq 4.y}
\lf\|VL^{-1}(a)\r\|_{L^2(U_i(B))}\le C2^{-i\tz}|B|^{\frac12}\|\chi_B\|_{\vp}^{-1}.
\end{align}
For any $i\in\{0,\,\ldots,\,10\}$, since $VL^{-1}$ is bounded on $L^2(\rn)$,
we know that
\begin{equation}\label{eq 4.x}
\lf\|VL^{-1}(a)\r\|_{L^2(U_i(B))}\ls\|a\|_{L^2(B)}\ls|B|^{1/2}\|\chi_B\|_{\vp}^{-1}.
\end{equation}
For any $i\in\nn$ and $i\geq 11$, from the fact that
$L^{-1}=\int_0^\fz e^{-tL}\,dt$, we deduce that
\begin{align}\label{eq 4.4x}
\lf\|VL^{-1}(a)\r\|_{L^2(U_i(B))}
&\le\lf\|\int_0^{r_B^2}V(\cdot) e^{-tL}(a)(\cdot)\,dt\r\|_{L^2(U_i(B))}\noz\\
&\hs+\lf\|\int_{r_B^2}^\fz V(\cdot) e^{-tL}(a)(\cdot)\,dt\r\|_{L^2(U_i(B))}\noz\\
&=:{\rm I}_i+{\rm II}_i.
\end{align}

We first estimate ${\rm I}_i$. By the Minkowski inequality,
we find that
\begin{align}\label{eq 4.4}
{\rm I}_i
&=\lf[\int_{U_i(B)}\lf|\int_0^{r_B^2}V(x)e^{-tL}(a)(x)\,dt\r|^2\,dx\r]^{1/2}\noz\\
&\le\int_0^{r_B^2}\lf[\int_{U_i(B)}\lf|V(x)e^{-tL}(a)(x)\r|^2\,dx\r]^{1/2}\,dt.
\end{align}
From the Minkowski inequality and the fact that $\supp a\st B$ (see Definition \ref{def atom-2}),
we further deduce that,
\begin{align}\label{eq 4.5x}
&\lf[\int_{U_i(B)}\lf|V(x)e^{-tL}(a)(x)\r|^2\,dx\r]^{1/2}\noz\\
&\hs=\lf[\int_{U_i(B)}\lf|\int_B V(x)K_t(x,\,y)a(y)\,dy\r|^2\,dx\r]^{1/2}\noz\\
&\hs\le\int_B |a(y)|\lf[\int_{U_i(B)}\lf|V(x)K_t(x,\,y)\r|^2\,dx\r]^{1/2}\,dy\noz\\
&\hs\le\int_B |a(y)|\lf[\int_{|x-y|\geq 2^{i-2}r_B}\lf|V(x)K_t(x,\,y)\r|^2\,dx\r]^{1/2}\,dy.
\end{align}
By applying Lemma \ref{lem 4.1}(i) with $k=0$ and $p=2$ in \eqref{eq 4.10y},
we obtain
\begin{align}\label{eq 4.5b}
&\lf[\int_{|x-y|\geq 2^{i-2}r_B}\lf|V(x)K_t(x,\,y)\r|^2\,dx\r]^{1/2}\noz\\
&\hs=\lf[\int_{|x-y|\geq 2^{i-2}r_B}\lf|V(x)K_t(x,\,y)\r|^2
e^{\az\frac{|x-y|^2}{t}}e^{-\az\frac{|x-y|^2}{t}}\,dx\r]^{1/2}\noz\\
&\hs\le e^{-\frac{\az}{8}\frac{2^{2i}r_B^2}{t}}
\lf[\int_{|x-y|\geq 2^{i-2}r_B}\lf|V(x)K_t(x,\,y)\r|^2
e^{\az\frac{|x-y|^2}{t}}\,dx\r]^{1/2}\noz\\
&\hs\ls e^{-\frac{\az}{8}\frac{2^{2i}r_B^2}{t}}\frac{1}{t^{1+n/4}},
\end{align}
where $\az\in(0,\,\fz)$ is as in \eqref{eq 4.10}.
By this, \eqref{eq 4.5x}, the H\"{o}lder inequality and Definition \ref{def atom-2},
we conclude that
\begin{align*}
&\lf[\int_{U_i(B)}\lf|V(x)e^{-tL}(a)(x)\r|^2\,dx\r]^{1/2}\\
&\hs\ls\int_B|a(y)|\frac{1}{t^{1+n/4}}e^{-\frac{\az}{8}\frac{2^{2i}r_B^2}{t}}\,dy
\ls t^{-(1+\frac{n}4)}e^{-\frac{\az}{8}\frac{4^ir_B^2}{t}}\|a\|_{L^2(B)}|B|^{1/2}\\
&\hs\ls t^{-(1+\frac{n}4)}e^{-\frac{\az}{8}\frac{4^ir_B^2}{t}}|B|\|\chi_B\|_{\vp}^{-1}.
\end{align*}
This, combined with \eqref{eq 4.4}, implies that
\begin{align}\label{eq 4.5}
{\rm I}_i
&\ls\int_0^{r_B^2} t^{-(1+\frac{n}4)}e^{-\frac{\az}{8}\frac{4^ir_B^2}{t}}
\|\chi_B\|_{\vp}^{-1}\,dt\noz\\
&\ls|B|\|\chi_B\|_{\vp}^{-1}\int_0^{r_B^2}\lf(\frac{t}{4^ir_B^2}\r)^N t^{-(1+\frac{n}{4})}\,dt
\ls 2^{-2iN}|B|^{1/2}\|\chi_B\|_{\vp}^{-1},
\end{align}
where $N$ is a positive constant large enough such that $N>\frac{n}{4}$,
which is determined later.

For ${\rm II}_i$,
from the Minkowski inequality, we deduce that
\begin{align}\label{eq 4.5y}
{\rm II}_i
&=\lf[\int_{U_i(B)}\lf|\int_{r_B^2}^\fz V(x)e^{-tL}(a)(x)\,dt\r|^2\,dx\r]^{1/2}\noz\\
&\le\int_{r_B^2}^\fz\lf[\int_{U_i(B)}\lf|V(x)e^{-tL}(a)(x)\r|^2\,dx\r]^{1/2}\,dt.
\end{align}
Moreover, by ${\rm (i)}$ of Definition \ref{def atom-2}, we have
\begin{equation*}
e^{-tL}(a)=e^{-tL}(L^M(b))=L^M e^{-tL}(b)=(-1)^M\frac{\partial^M}{\partial t^M}e^{-tL}(b).
\end{equation*}
By this, the Minkowski inequality, Lemma \ref{lem 4.1}(i)
and an argument similar to that used in \eqref{eq 4.5b}, we know that
\begin{align*}
&\lf[\int_{U_i(B)}\lf|V(x)e^{-tL}(a)(x)\r|^2\,dx\r]^{1/2}\\
&\hs=\lf[\int_{U_i(B)}\lf|\int_B V(x)
\lf(\frac{\partial^M}{\partial t^M}K_t(x,\,y)\r)b(y)\,dy\r|^2\,dx\r]^{1/2}\\
&\hs\le\int_B|b(y)|\lf[\int_{|x-y|\geq 2^{i-2}r_B}\lf|V(x)
\lf(\frac{\partial^M}{\partial t^M}p_t(x,y)\r)\r|^2\,dx\r]^{1/2}\,dy\\
&\hs\ls\int_B|b(y)|t^{-(1+\frac{n}{4}+M)}e^{-c\frac{4^ir_B^2}{t}}\,dy
\ls t^{-(1+\frac{n}4+M)}e^{-c\frac{4^ir_B^2}{t}}\|b\|_{L^2(B)}|B|^{1/2}\\
&\hs\ls t^{-(1+\frac{n}4+M)}e^{-c\frac{4^ir_B^2}{t}}r_B^{2M}|B|\|\chi_B\|_{\vp}^{-1}.
\end{align*}
From this and \eqref{eq 4.5y}, we deduce that
\begin{align*}
{\rm II}_i
&\ls |B|\|\chi_B\|_{\vp}^{-1}4^{-iM}\int_{r_B^2}^\fz t^{-(1+\frac{n}4)}
\lf(\frac{4^ir_B^2}{t}\r)^Me^{-c\frac{4^ir_B^2}{t}}\,dt\\
&\ls 2^{-2iM}|B|^{1/2}\|\chi_B\|_{\vp}^{-1}.
\end{align*}
By this, \eqref{eq 4.5} and \eqref{eq 4.4x}, we find that
\begin{align*}
\lf\|VL^{-1}(a)\r\|_{L^2(U_i(B))}
&\ls \lf[2^{-2iN}+2^{-2iM}\r]|B|^{1/2}\|\chi_B\|_{\vp}^{-1}\\
&\ls 2^{-i\tz}|B|^{1/2}\|\chi_B\|_{\vp}^{-1},
\end{align*}
where $\tz:=\min\{2N,\,2M\}$. By choosing $N>\frac{n}2(\frac{1}{p_-}-\frac12)$ and
the fact that $M\in\nn\cap(\frac{n}2[\frac{1}{p_-}-\frac12],\fz)$, we find
that $\tz\in(n[\frac{1}{p_-}-\frac12],\fz)$.
Hence, \eqref{eq 4.y} holds true.
This finishes the proof of \eqref{eq-t2}.

By Lemma \ref{lem 4.1}, the proofs of \eqref{eq-t3} and \eqref{eq-t1} are totally similar to that
of \eqref{eq-t2}, the details being omitted.
Hence, we complete the proof of Theorem \ref{thm-4}.
\end{proof}

\noindent
{\bf Acknowledgments.}
Junqiang Zhang is very grateful to his advisor Professor Dachun Yang
for his guidance and encouragements.
Junqiang Zhang is
supported by the Fundamental Research Funds for the Central Universities
(Grant No. 2018QS01) and the National
Natural Science Foundation of China (Grant No. 11801555).
Zongguang Liu is supported by the National
Natural Science Foundation of China (Grant No. 11671397).

\bibliographystyle{amsplain}

\end{document}